\documentclass[a4paper,12pt]{article}
\usepackage{CJKutf8}
\usepackage{CJK}
\usepackage{amsfonts}
\usepackage{pifont}
\usepackage{bm}
\usepackage{latexsym,amsmath,amssymb,cite,amsthm}
\usepackage{color,eucal,enumerate,mathrsfs}
\usepackage[normalem]{ulem}
\usepackage{amsmath}
\usepackage{hyperref}
\usepackage[pagewise]{lineno}

\numberwithin{equation}{section}
\theoremstyle{plain}
\newtheorem{theorem}{Theorem}[section]
\newtheorem{corollary}[theorem]{Corollary}
\newtheorem{lemma}[theorem]{Lemma}
\newtheorem{proposition}[theorem]{Proposition}
\theoremstyle{definition}
\newtheorem{definition}[theorem]{Definition}
\theoremstyle{remark}
\newtheorem{assumption}[theorem]{Assumption}
\newtheorem{remark}[theorem]{Remark}

\newcommand{\lmt}[2]{\mathop{\lim}_{{#1} \rightarrow {#2}} }

\newcommand{\lip}[1]{{\mathrm{lip}}({#1})}
\newcommand{\loclip}[2]{{\mathrm{lip}}_{#2}({#1})}

\newcommand{\esup}[1]{{\mathrm{ess~sup}}{#1}}

\newcommand{\lmts}[2]{\mathop{\overline{\lim}}_{{#1} \rightarrow {#2}} }

\newcommand{\Ric}{{\rm{Ricci}}}

\newcommand{\bRicn}{{\bf Ricci}_N}
\newcommand{\tr}{{\rm{tr}}}
\renewcommand{\H}{{\mathrm{Hess}}}

\newcommand{\mm}{\mathfrak m}
\newcommand{\mmae}{\mathfrak m \text{-a.e.}}
\newcommand{\ms}{(X,\d,\mm)}
\newcommand{\cd}{{\rm CD}^*(K, \infty)}
\newcommand{\cdkn}{{\rm CD}(K, N)}
\newcommand{\rcdkn}{{\rm RCD}^*(K, N)}
\newcommand{\rcd}{{\rm RCD}(K, \infty)}

\newcommand{\R}{\mathbb{R}}

\newcommand{\Id}{{\rm Id}}
\newcommand{\supp}{\mathop{\rm supp}\nolimits}   %%\newcommand{\span}{\mathop{\rm span}\nolimits}   %
\newcommand{\Lip}{\mathop{\rm Lip}\nolimits}
\renewcommand{\d}{{\mathrm d}}

\newcommand{\D}{{\mathrm D}}
\newcommand{\restr}[1]{\lower3pt\hbox{$|_{#1}$}}
\newcommand{\la}{{\langle}}
\newcommand{\ra}{{\rangle}}

\newcommand{\nchi}{{\raise.3ex\hbox{$\chi$}}}

\title{\large{\bf  Conformal transformation on metric measure spaces}
}
\begin{document}
\author{Bang-Xian Han \thanks
{Institute for Applied Mathematics, University of Bonn,  han@iam.uni-bonn.de}
}

\date{\today}
\maketitle

\begin{abstract}
We  study several problems concerning conformal transformation on metric measure spaces, including the Sobolev space, the differential structure and the curvature-dimension condition under conformal transformations.   This is the first result about preservation of lower curvature bounds under perturbation, which is new even on Alexandrov spaces.
\end{abstract}



\tableofcontents

\section{Introduction}
Let $(M, g, {\rm Vol}_g)$ be a $n$-dimensional Riemannian manifold, $\Ric(\cdot, \cdot)$ be the Ricci curvature tensor on it, and let $w$ be a smooth function on $M$. Then  the corresponding Riemannian manifold under conformal transformation is defined by $(M, e^{2w}g, e^{nw}{\rm Vol}_g)$. We know that this transformation preserves the angle between tangent vectors,  and the following formula holds (see Theorem 1.159, \cite{B-E}).
\begin{equation}\label{eq:conformal}
\Ric'=\Ric-(n-2)(\D \d w -\d w \otimes \d w)+(-\Delta w-(n-2) |\d w|^2) g,
\end{equation}
where $\D \d w$ is the Hessian of $w$ and $\Delta$ is the Laplace-Beltrami operator on $(M, g)$. In particular, this formula can be used to study the lower Ricci curvature bound under conformal transformation.

The conformal transformation defined as above plays an important role in differential geometry, and has potential applications in non-smooth setting.  Similar to the construction of cone, sphere and warped product (see \cite{K-C} and \cite{GH-W}),  conformal transformation on metric measure space can be defined in an intrinsic way.

Let $w, v$ be  bounded continuous functions on a metric measure space $\ms$. We can build a new metric measure space $M':=(X, \d',\mm')$ where
\begin{itemize}
\item [i)]  weighted measure $\mm'$ is defined by:
\[
\mm'=e^{v}\,\mm,
\]
\item [ii)] weighted metric $\d'$  is defined by:
\[
\d'(x, y)=\mathop{\inf}_{\gamma} \Big \{ \int_0^1 |\dot{\gamma_t}|e^{w(\gamma_t)}\,\d t: \gamma \in {\rm AC}([0,1],X), \gamma_0=x, \gamma_1=y \Big \}.
\]
\end{itemize}

On metric measure spaces, the notion of synthetic Ricci curvature bounds, or non-smooth curvature-dimension conditions, is proposed by Lott-Sturm-Villani (see \cite{Lott-Villani09} and \cite{S-O1} for ${\rm CD}(K, \infty)$ and $\cdkn$ conditions) and Bacher-Sturm (see \cite{BS-L} for ${\rm CD}^*(K, N)$ condition). More recently, based on some new results about the Sobolev spaces on metric measure space (see \cite{AGS-C}),  $\rcd$ and $\rcdkn$ conditions, which are refinements of curvature-dimension conditions,  are proposed by Ambrosio-Gigli-Savar\'e (see \cite{AGS-M} and \cite{AGMR-R}). Moreover, the non-smooth Bakry-\'Emery theory, which offers  equivalent characterization of
  $\rcd$ and $\rcdkn$ conditions, has been studied in  \cite{AGS-B}, \cite{AMS-N} and \cite{EKS-O}. These Riemannian curvature-dimension conditions are stable with respect to the measured Gromov-Hausdorff convergence, and cover the cases of Riemannian manifolds, smooth metric measure spaces, Alexandrov spaces and their limits. 
  
 Then we have some natural questions:
 \begin{itemize}
 \item [1)] What is the conformal transformation on $\rcdkn$ spaces?  How to characterize it?
 
\item [2)] Do we have formula   \eqref{eq:conformal}  on $\rcdkn$ spaces? 

\item [3)]  Can we  study the curvature-dimension condition under conformal transformation by using \eqref{eq:conformal}?
 \end{itemize}

To answer these questions, we divide this paper into two parts. In the first part,  we  study the  Sobolev space  and the differential structure of metric measure spaces under conformal transformation. These are the basic tools to study  metric measure space.  The results we have obtained are  useful to study curvature-dimension condition, geometric flows, sectional curvature, etc. In summary, we prove the following results on non-smooth metric measure spaces:

\begin{itemize}
\item [1)](Sobolev space, Proposition \ref{lemma-1}) The Sobolev spaces  $W^{1,2}(M)$ and $W^{1,2}(M')$ coincide as sets and
\[
|\D f|_{M'}=e^{-w}|\D f|_M,~~~\mmae
\]
for any $f\in W^{1,2}(M)$. 
\item [2)](Laplacian, Proposition \ref{lemma-2}) For any $f \in {\rm D}({\bf \Delta}')$, we have
 \[
{\bf \Delta}'f=e^{v-2w}\big{(}{\bf \Delta}f+\Gamma(v-2w,f)\,\mm\big{)},
\]
where $\Gamma(\cdot, \cdot)$ is the  carr\'e du champ operator induced by the weak upper gradients.
\item [3)](Tangent vector, Proposition \ref{lemma-3}) 
$$
\nabla' f=e^{-2w}\nabla f,~~~\mm-\text{a.e.}
$$
and 
$$
\la X, Y \ra_{M'}=e^{2w} \la X, Y \ra,~~~\mm-\text{a.e.}.
$$

The second  formula can be seen as an equivalent characterization of conformal transformation on infinitesimally Hilbertian spaces.
\item [4)] (Hessian, Proposition \ref{Hessian})
\begin{eqnarray*}
|\H'_f|^2_{\rm HS}=e^{-4w}\big{(}|\H_f|^2_{\rm HS}&+&2\Gamma(f)\Gamma(w)+({\dim_{\rm loc}}-2) \Gamma(f,w)^2\\
&-&2\Gamma(w,\Gamma(f))+2\Gamma(f,w)\tr \H_f \big{)}
\end{eqnarray*}
and
\[
\tr \H'_f=e^{-2w} \big{(} \tr \H_f+({\dim_{\rm loc}}-2) \Gamma(f,w) \big{)}
\]
hold $\mm$-a.e. .
\item [5)] (Covariant derivative, Proposition \ref{Cov})
\begin{eqnarray*}
\nabla' X:'( Y \otimes Z)
&=& \nabla X:( Y \otimes Z)-\la Y, \nabla w \ra \la X, Z \ra -\la Z, \nabla w \ra \la X, Y \ra  \\
&&~~~~~~~~+\la X, \nabla w \ra \la Y, Z\ra 
\end{eqnarray*}
\end{itemize}

We say that   a metric measure space $M=(X ,\d, \mm)$ has  Sobolev-to-Lipschitz property if for any $f\in W^{1,2}(X)$  with $|\D f| \in L^\infty$, there exits a Lipschitz continuous function  $\bar{f}$ such that $f=\bar{f}$ $\mm$-a.e. and  $\Lip (\bar{f})=\esup ~{|\D f|}$. 
The Sobolev-to-Lipschitz property is an important property which links metric  and differential structure (see \cite{G-S} and \cite{GH-W}). It is also a pre-requisite to apply Bakry-\'Emery theory on metric measrue space (see \cite{AGS-B}).  In Proposition \ref{lemma-4}, we prove that the conformal transformation preserves Sobolev-to-Lipschitz property.

\begin{proposition}[Sobolev-to-Lipschitz property]
Let $M$ be a $\rcdkn$ metric measure space, where $N <\infty$. The space $M'$ is  constructed by conformal transformation with continuous and bounded conformal factors. Then $M'$  satisfies the Sobolev-to-Lipschitz property.
\end{proposition}

In the second part of this paper,  we study the curvature dimension condition  under  conformal transformation.  This problem is difficult to handle using  Lott-Sturm-Villani's original definition. One possible method to conquer this difficulty  is to construct a  formula similar to \eqref{eq:conformal}.  In  \cite{S-R} Sturm defines an abstract Ricci tensor,  and studies its conformal transformation under some  smoothness assumptions. In this work we apply the results about differential structure on $\rcdkn$  spaces, which is studied in \cite{G-N} (and \cite{H-R}), to prove Sturm's result on $\rcdkn$ space.  In Theorem \ref{th-conformal} we extend the formula \eqref{eq:conformal} to $\rcdkn$ spaces. As an application, we  obtain an estimate of the curvature-dimension condition under conformal transformation. 

Now we briefly explain the proof of the main theorem.
 Firstly,  based on  results about the Sobolev space under conformal transformation,  we study the non-smooth (co)tangent modules developed by Gigli (in \cite{G-N})  under conformal transformation. Secondly, we  compute the measure-valued  Laplacian  (see \cite{G-O} and \cite{S-S}). Then we prove in Proposition \ref{lemma-4} that the conformal transformation preserves Sobolev-to-Lipshitz property, so we can apply Bakry-\'Emery's theory. In \cite{G-N} and \cite{H-R}, the measure-valued Ricci tensor is defined as
 \begin{equation}\label{eq:riccin}
\bRicn(\nabla f, \nabla f):={\bf \Gamma}_2(f) -\Big{(} |\H_f|^2_{\rm HS}+\frac1{N-{\dim_{\rm loc}}}(\tr \H_f -\Delta f)^2 \Big{)} \, \mm,
\end{equation}
where  $\dim_{\rm loc}$ is the local dimension.
It is proved in \cite{G-N} and \cite{H-R} that $\bRicn\geq K$ is equivalent to Bochner inequality and $\rcdkn$ condition. 
Combining with our results on  Hessian and its Hilbert-Schmidt norm, we show that $\dim_{\rm loc}$ is invariant under conformal transformation,  and the Ricci tensor under conformal transformation is well-defined, see Theorem \ref{th-conformal} for the transformation formula. Then we  obtain the estimate of the curvature-dimension condition under conformal transformation in Corollary \ref{coro-conformal} and Corollary \ref{coro-conformal-2}.
In particular, for  the  transformation $\ms \mapsto (X, e^{w}\d, e^{Nw}\,\mm)$ on $\rcdkn$ space, we  obtain in Corollary \ref{coro-N} a non-smooth version of the formula \eqref{eq:conformal}. We remark that our result about the lower curvature estimate, even on Alexandrov space, is new and optimal.

The paper is organized as follows. In Section 2 we  introduce the notions of Sobolev space, non-smooth Bakry-\'Emery theory, the tangent/cotangent module and analytic dimension of metric measure spaces.  In Section 3 we  study the conformal transformation on metric measure spaces, the Sobolev space as well as the differential structures under conformal transformation. All these objects are considered in pure intrinsic ways. We study the Ricci tensor and obtain a generalization of the formula \eqref{eq:conformal} on $\rcdkn$ spaces. Then we obtain a precise   estimate of $N$-Ricci curvature under conformal transformation.

\noindent \textbf{Acknowledgement}: The author acknowledges the support of the HCM fellowship. He thanks Prof. Karl-Theodor Sturm for proposing this topic and valuable advice, and he thanks Anna Mkrtchyan for reading the preliminary version of the manuscript. The author also wants to thank the reviewer for helpful remarks and  suggestions. \\

\section{Preliminaries}
Basic assumptions on metric measure spaces are the following. 
\begin{assumption}\label{assumption}
Let $M:=\ms$ be a metric measure space. We assume that 
\begin{itemize}
\item [a)] $(X ,\d)$ is a complete separable geodesic space.
\item [b)] $\mm$ is a $\d$-Borel measure  satisfying 
\[
\supp \mm=X,~~~~\mm(B_r(x_0)) < c_1\exp{(c_2 r^2)}~~~\text{for every}~~r>0,
\]
 for some constants $c_1, c_2 \geq 0$ and a point $x_0 \in X$.
 \item [c)] $M$ is an infinitesimally Hilbertian space.
\end{itemize} 
\end{assumption}
 Important examples satisfying Assumption \ref{assumption} are  $\rcdkn$ metric measure spaces, where $K \in \R$ and $N \in [1, \infty]$ (it is $\rcd$ when $N=\infty$). Both $\rcd$ and $\rcdkn$ conditions are refinements of the curvature-dimensions proposed by Lott-Sturm-Villani (see \cite{Lott-Villani09} and \cite{S-O1} for $\cd$) and Bacher-Sturm (see \cite{BS-L} for ${\rm CD}^*(K, N)$). The  inclusions of these curvature dimension conditions are
   \[
\rcdkn \subset {\rm CD}^*(K, N) ~~\text{and}~~ \rcd \subset \cd,
 \]
 and
 \[
\rcdkn \subset \rcd ~~\text{and}~~{\rm CD}^*(K, N) \subset \cd.
 \]

More details about the curvature dimension condition $\rcdkn$ can be found in \cite{AGMR-R}, \cite{AGS-M} and \cite{EKS-O}.

Let $f : X \mapsto \mathbb{R}$. The local Lipschitz constant $\lip{f}: X \mapsto [0, \infty]$ is defined as
\[
\lip{f}(x):=
\left 
\{ \begin{array}{ll}
 \mathop{\overline{\lim}}_{y \to x}\frac{|f(y) -f(x)|}{\d(x, y)}~~\text{if}~x~\text{is not isolated},\\
0, ~~~~~~~~~~~~~~~~~~~~~~~~~~~~~\text{otherwise}.
\end{array} 
\right .\]
The (global) Lipschitz constant is defined in the usual way as
\[
\Lip(f):= \mathop{\sup}_{x \neq y} \frac{|f(y)-f(x)|}{\d(x,y)}.
\]
Since $(X, \d)$ is a geodesic space, we know $\Lip(f)=\mathop{\sup}_{x}  \lip{f}(x)$.

Now we introduce the the Sobolev space $W^{1,2}(M)$, which is firstly introduced by Cheeger in \cite{C-D}.
In \cite{AGS-C}, a new characterization has been proposed. We say that $f\in L^2(X, \mm)$ is a Sobolev function in $W^{1,2}(M)$ if there exists a sequence of Lipschitz functions $\{f_n\} \subset L^2$,  such that $f_n \to f$ and $\lip{f_n} \to G$ in $L^2$ for some $G \in L^2(X, \mm)$. It is known that there exists a minimal function $G$ in $\mm$-a.e. sense. We call this  $G$ the minimal weak  upper gradient (or weak gradient for simplicity) of the function $f$, and denote it by $|\D f|$ or $|\D f|_M$ to indicate which space we are considering.

We equip $W^{1,2}(M)$ with the norm
\[
\|f\|^2_{W^{1,2}\ms}:=\|f\|^2_{L^2(X,\mm)}+\||\D f|\|^2_{L^2(X,\mm)}.
\]
If $W^{1,2}\ms$ is a Hilbert space, we call $\ms$  infinitesimally Hilbertian (which is introduced in \cite{G-O}).

As a consequence of the definition above, we have the lower semi-continuity: if $\{f_n\}_n\subset W^{1,2}\ms$ is a sequence converging to some $f$ in $\mm$-a.e. sense such that $\{|\D f_n|\}_n$ is bounded in $L^2(X,\mm)$, then $f\in W^{1,2}\ms$ and
\[
|\D f|\leq G,\qquad\mm\text{-a.e.},
\]
for every $L^2$-weak limit $G$ of some subsequence of $\{|\D f_n|\}_n$. Furthermore, we have the following proposition.

\begin{proposition}[see \cite{AGS-C}]\label{prop-density}
Let $\ms$ be a metric measure space.  Then the Lipschitz functions are dense in energy in $W^{1,2}(M)$ in following sense. For any $f\in W^{1,2}(M)$, there is a sequence of Lipschitz functions $\{f_n \}_n\subset  L^2(X,\mm)$ such that $f_n \rightarrow f$ and $\lip{ f_n} \rightarrow |\D f|$ in $L^2$.
\end{proposition}

Next, we introduce some  basic notions and facts about `tangent/cotangent vector field' in non-smooth setting, more details can be found in \cite{G-N}.

\begin{definition}[$L^2$-normed $L^\infty$-module]\label{L2norm}

Let $M=\ms$ be a metric measure space. A $L^2$-normed $L^\infty(X,\mm)$-module is a Banach space $({\bf B}, \| \cdot \|_{\bf B})$ equipped with a bilinear map
\begin{eqnarray*}
L^\infty(X,\mm) \times {\bf B} &\mapsto&  {\bf B},\\
 (f,v) &\mapsto& f \cdot v
\end{eqnarray*}
such that
\begin{eqnarray*}
(fg) \cdot v &=& f \cdot  (g \cdot v),\\
 {\bf 1} \cdot v &=& v
\end{eqnarray*}
for every $v \in {\bf B}$ and $f, g \in L^\infty(M)$, where ${\bf 1} \in L^\infty(X,\mm)$ is identically equal to 1 on $X$, and a `pointwise norm' $|\cdot|: {\bf B} \mapsto  L^2(X,\mm)$ which maps $v \in {\bf B}$ to  a non-negative $L^2$-function
such that
\begin{eqnarray*}
\|v\|_{\bf B} &=&  \| |v| \|_{L^2}\\
|f\cdot v| &=& |f||v|,~~~ \mmae.
\end{eqnarray*}
for every $f \in L^\infty(X,\mm)$  and $v \in  {\bf B}$.
\end{definition}

%%%%%%%%%%%%%%%%%%%%%%%%å·¥ä½è³æ­¤
 We define the `Pre-Cotangent Module' $\mathcal{PCM}$ as the set consisting of  elements of  form  $\{ (B_i, f_i )\}_{i  \in \mathbb{N}}$, where $\{B_i\}_{i \in \mathbb{N}}$ is a Borel partition of $X$, and $\{f_i\}_i$ are Sobolev functions with $\sum_i \int_{B_i} |\D f_i|^2 <\infty$.

We define an equivalence relation on $\mathcal{PCM}$ via
\[
 \{ (A_i, f_i )\}_{i  \in \mathbb{N}} \thicksim \{ (B_j, g_j )\}_{j  \in \mathbb{N}}~~~\text {if}~~~~
 |\D(g_j-f_i)|=0, ~~\mm-\text {a.e. on}~ A_i \cap B_j.
\]
We denote the equivalence class of $\{ (B_i, f_i )\}_{i  \in \mathbb{N}}$ by $[ (B_i, f_i )]$. In particular, we call $[(X, f)]$ the differential of a Sobolev function $f$ and denote it by $\d f$.

Then we define the following operations:
\begin{itemize}\label{premodule}
\item [1)] $[ (A_i, f_i )]+[ (B_i, g_i )]:=[ (A_i \cap B_j, f_i+g_j )]$,

\item [2)] Multiplication by scalars: $\lambda [ (A_i, f_i )]:= [ (A_i, \lambda f_i )]$,

\item [3)] Multiplication by simple functions: $(\sum_j \lambda_j \nchi_{B_j})  [ (A_i, f_i )]:=[ (A_i \cap B_j, \lambda_j f_i )]$,
\item [4)] Pointwise norm: $|[ (A_i, f_i )]|:=\sum_i \nchi_{A_i} |\D f_i|$,
\end{itemize}
where $\nchi_A$ is the characteristic function on the set $A$.
%%%%%%%%%%%%%%%%%%%%%%%%%%%%%%%%%%%%%%%%
\bigskip

It can be seen that all the operations above are continuous on $\mathcal{PCM}/ \thicksim$ with respect to the norm $\| [ (A_i, f_i )] \|:= \sqrt{\int|[ (A_i, f_i )]|^2\,  \mm}$ and the $L^\infty(M)$-norm on the space of simple functions. Therefore, we can extend them to the completion of $(\mathcal{PCM}/\thicksim, \| \cdot \|)$ and we denote this completion by $L^2(T^*M)$. As a consequence of our definition, we can see that $L^2(T^* M)$ is the $\| \cdot \|$ closure of $\{\sum_{i \in I} a_i \d f_i: |I|<\infty,  a_i \in L^\infty(M), f_i \in W^{1,2} \}$ (see Proposition 2.2.5 in \cite{G-N} for a proof).  It can also be seen from the definition and the infinitesimal Hilbertianity assumption on $M$ that $L^2(T^*M)$ is a Hilbert space equipped with the inner product induced by $\| \cdot \|$. Moreover, $(L^2(T^*M), \| \cdot \|, |\cdot |)$ is a $L^2$-normed module according to Definition \ref{L2norm}, which we shall call cotangent module of $M$.

%%%%%%%%%%%%%%%%%%%%
%It is known from/ It can be seen/proved
%%%%%%%%%%%%%%%%%%%%%
We then define the tangent module $L^2(TM)$ as  $\mathrm{Hom}_{L^\infty(M)}(L^2(T^*M), L^1(M))$. In other words, $T \in L^2(T^* M)$ if it is a continuous linear map from $L^2(T^*M)$  to $L^1(M)$ viewed as Banach spaces satisfying the homogeneity:
\[
T(f v)=f  T(v), ~~\forall v \in L^2(T^*M),~~ f \in L^\infty(M).
\]

It can be seen that $L^2(TM)$ has a natural $L^2$-normed $L^\infty(M)$-module structure and  is isometric to $L^2(T^*M)$ both as a module and as a Hilbert space. We denote the corresponding element of $\d f$ in $L^2(TM)$ by $\nabla f$ and call it the gradient of $f$. By Riesz theorem for Hilbert modules (see Chapter 1 of \cite{G-N}), we know that $\d f(\nabla f):=\nabla f( \d f)=|\D f|^2$. The natural pointwise norm on $L^2(TM)$ (we also denote it by $| \cdot |$) satisfies $|\nabla f|=|\d f|=|\D f|$. It can also be seen that  $\{\sum_{i \in I} a_i  \nabla f_i: |I|<\infty,  a_i \in L^\infty(M), f_i \in W^{1,2} \}$ is a dense subset in $L^2(T M)$.

On an infinitesimally Hilbertian space, we have a natural `carr\'e du champ' operator $\Gamma(\cdot, \cdot): [W^{1,2}(M)]^2 \mapsto L^1(M)$ defined by
\[
\Gamma(f, g):= \frac14 \Big{(}|\D (f+g)|^2-|\D (f-g)|^2\Big{)}.
\]
We denote $\Gamma(f,f)$ by $\Gamma(f)$.

Then we have a pointwise inner product $\la \cdot, \cdot \ra: [L^2(T^*M)]^2 \mapsto L^1(M)$ satisfying
\[
\la \d f, \d g \ra:= \Gamma(f, g)
\]
for $f, g \in W^{1,2}(M)$.  We know also from Riesz theorem that the gradient $\nabla g$  is exactly the element in $L^2(TM)$ such that $\nabla g (\d f)=\la \d f, \d g \ra$, $\mm$-a.e. for every $f \in W^{1,2}(M)$. Therefore, $L^2(TM)$ inherits a pointwise inner product from $L^2(T^*M)$ and we still use $\la \cdot, \cdot \ra$ to denote it.

It is known from  \cite{G-N} that the following basic calculus rules hold in $\mm$-a.e. sense.

We have that
\begin{itemize}
\item [i)] $\d (fg)=f\d g +g\d f$,
\item [ii)] $\d (\varphi \circ f)=\varphi'\circ f \d f$,
\end{itemize}
for every $f, g \in W^{1,2}(M)$, and every smooth $\varphi: \R \mapsto \R$  with bounded derivative.

We then define the Laplacian by duality/integration by part. 
The space of finite  Borel measures on $M$, equipped with the total variation norm $\| \cdot \|_{\rm TV}$, is denoted by ${\rm Meas}(M)$.

\begin{definition}
[Measure valued Laplacian, \cite{G-O, G-N}]
The space ${\rm D}({\bf \Delta}) \subset  W ^{1,2}(M)$ is the space of $f \in  W ^{1,2} (M)$ such that there is a measure ${\bf \mu}\in {\rm Meas}(M)$ satisfying
\[
\int \varphi \,{\mathbf \mu}= -\int \Gamma (\varphi,  f ) \,   \mm, ~~\forall \varphi: M \mapsto  \R, ~~ \text{Lipschitz with bounded support}.
\]
In this case the measure $\mu$ is unique and we denote it by ${\bf \Delta} f$. If ${\bf \Delta} f \ll m$, we denote its density  with respect to $\mm$ by $\Delta f$.
\end{definition}

It is proved in \cite{G-O} that the following  rules hold for the Laplacian:
\begin{itemize}
\item [i)] ${\bf \Delta}(fg)=f{\bf \Delta} g +g{\bf \Delta} f+2\Gamma(f,g)\,\mm$,
\item [ii)] ${\bf \Delta} (\varphi \circ f)=\varphi'\circ f {\bf \Delta} f+\varphi''\circ f \Gamma(f)\,\mm$,
\end{itemize}
for every $f, g \in {\rm D}({\bf \Delta}) \cap L^\infty(M)$, and every smooth $\varphi: \R \mapsto \R$  with bounded first and second derivatives.

We define  ${\rm TestF}(M) \subset W^{1,2}(M)$, the space of test functions  as
\[
{\rm TestF}(M):= \Big\{f \in {\rm D} ({\bf \Delta}) \cap L^\infty: |\D f|\in L^\infty~~ {\rm and}~~~ \Delta f  \in W^{1,2}(M) \Big\}.
\]
It is known from \cite{S-S}  that ${\rm TestF}(M)$ is an algebra and it is  dense in $W^{1,2}(M)$ when $M$ is $\rcd$. In particular, we know the space of test vectors $\{\sum_{i \in I} a_i  \nabla f_i: |I|<\infty,  a_i \in L^\infty(M), f_i \in {\rm TestF}(M) \}$ is dense in $L^2(T M)$.

We also have the following lemma.

\begin{lemma}[\cite{S-S}]\label{self-improve}
Let $M=\ms$ be a $\rcd$  space, $f\in {\rm TestF}(M)$ and $\Phi \in C^\infty(\R)$
 with $\Phi(0)=0$. Then $\Phi \circ f \in {\rm TestF}(M)$.
\end{lemma}

\bigskip

It is proved in \cite{S-S} that $\Gamma(f, g)  \in {\rm D}({\bf \Delta})\subset W^{1,2}(M)$ for any $f, g \in {\rm TestF}(M)$. Hence we can define the Hessian and ${\bf \Gamma}_2$ operator as follows. 

Let $f \in {\rm TestF}(M)$. We define the Hessian $\H_f: \{ \nabla g: g \in {\rm TestF}(M)\}^2 \mapsto L^0(M)$ by
\[
2\H_f(\nabla g,\nabla h)=\Gamma ( g,  \Gamma ( f,  h )) +\Gamma(h,  \Gamma(f, g))-\Gamma( f, \Gamma( g, h)),
\]
 for any $g, h \in {\rm TestF}(M)$. It has been proven in \cite{G-N} that $\H_f$ can be extended to a symmetric $L^\infty(M)$-bilinear map on $L^2(TM)$ and continuous with values in $L^0(M)$.

Let $f,g \in {\rm TestF}(M)$. We define the measure valued operator ${\bf \Gamma}_2(\cdot, \cdot)$ by
\[
{\bf \Gamma}_2(f,g):=\frac12 {\bf \Delta} \Gamma (f,  g) -\frac12 \big{(}\Gamma (f,  \Delta g)+\Gamma ( g,  \Delta f )\big{)}\, \mm,
\]
and we put ${\bf \Gamma}_2(f):={\bf \Gamma}_2(f,f)$.

Then we can characterize the curvature-dimension condition using non-smooth Bakry-\'Emery theory. We recall that  $\rcd$ and $\rcdkn$ spaces are infinitesimally Hilbertian $\cd$ and $\cdkn$ spaces.
We also  recall  the following terminology which is introduced in \cite{G-S}. We say that a metric measure space $M=(X ,\d, \mm)$ has  Sobolev-to-Lipschitz property if for any function $f\in W^{1,2}(X)$  such that $|\D f| \in L^\infty$, we can find a Lipschitz continuous function  $\bar{f}$ such that $f=\bar{f}$ $\mm$-a.e. and  $\Lip (\bar{f})=\esup ~{|\D f|}$.

We have the following definition/proposition. This weak Bochner's inequality has been proposed in \cite{AGS-B} for every $K, N$, and it is proved to be equivalent to $\rcdkn$ condition in \cite{AGS-B} ($N=\infty$) and \cite{AMS-N},\cite{EKS-O} ($N<\infty$). We rewrite it according to the results in \cite{S-S} (see Lemma 3.2 and Theorem 4.1 there). 

\begin{proposition} [Bakry-\'Emery condition, \cite{AGS-B}, \cite{EKS-O}, \cite{S-S}]\label{becondition}
Let $M=\ms$ be  an  infinitesimal Hilbert space satisfying Assumption \ref{assumption} and  the Sobolev-to-Lipschitz property, ${\rm TestF}(M)$ is dense in $W^{1,2}(M)$. Then it is a $\rcdkn$ space with $K \in \R$ and $N \in [1, \infty]$ if and only if the Bochner's inequality 
\[
{\bf \Gamma}_2(f) \geq \Big {(} K |\D f|^2+ \frac1N (\Delta f)^2 \Big{)}\,\mm
\]
holds in weak sense for any $f \in {\rm TestF}(M)$, that is
\[
\int \Delta g\frac{|\D f|^2}2\,\d \mm\geq\int g\Big(K|\D f|^2+\Gamma(f,\Delta f)+\frac{(\Delta f)^2}N\Big )\,\d \mm
\]
\end{proposition}
for  any $g\in \{ g\in W^{1,2}(M): g, \Delta g\in L^\infty \}$.
\begin{remark}
In some articles (for example \cite{AGS-B}, \cite{EKS-O}), the  following property
\[
\d(x ,y)=\sup \big\{f(x)-f(y): f \in W^{1,2}(M) \cap C_b(M), |\D f| \leq 1,~~\mmae \big \}
\]
is needed.
It can be seen that  this property can be obtained from Sobolev-to-Lipschitz property by considering the functions $\{\d(z, \cdot) : z \in X \}$.
\end{remark}

Now we  turn to discuss the dimension of $M$ which is understood as the dimension of  $L^2(TM)$  (as a $L^\infty$-module).
Let $B$ be a Borel set. We denote by  $L^2(TM)\restr{B}$ the subset of $L^2(TM)$ consisting of those $v$ such that $\nchi_{B^c} v=0$.

\begin{definition}[Local independence] Let $B$ be a Borel set with positive measure. We say that $\{v_i\}_1^n \subset L^2(TM)$ is independent on $B$ if
\[
\sum_i f_i v_i = 0, ~~\mmae ~~\text{on}~~ B
\]
holds if and only if $f_i = 0$  $\mm$-a.e. on $B$ for each $i$.

\end{definition}

\begin{definition}[Local span and generators] Let $B$ be a Borel set in $X$ and $V:=\{v_i\}_{i \in I} \subset L^2(TM)$. The span of $V$ on $B$, denoted by ${\rm Span}_B(V)$, is the subset of $L^2(TM)\restr{B}$  with the following property: there exists a Borel decomposition $\{B_n\}_{n\in \mathbb{N}}$ of $B$,  families of vectors $ {\{v_{i,n}\}_{i=1}^{m_n} \subset B}$ and functions $\{f_{i,n}\}_{i=1}^{m_n} \subset L^\infty(M)$, $n=1,2,...,$ such that
\[
\nchi_{B_n}  v=\mathop{\sum}_{i=1}^{m_n} f_{i,n} v_{i,n}
\]
for each $n$.
We call the closure of ${\rm Span}_B(V)$  the space generated by $V$ on $B$.
\end{definition}

We say that $L^2(TM)$ is finitely generated if there exists $\{v_1,...,v_n\}$ spanning $L^2(TM)$ on $X$, and locally finitely generated if there is a (Borel) partition $\{E_i\}$ of $X$ such that $L^2(TM)\restr{E_i}$ is finitely generated for every $i \in \mathbb{N}$.

\begin{definition}[Local basis and dimension] We say that a finite set $\{ v_1,...,v_n\}$ is a basis on a Borel set $B$ if it is independent on $B$ and ${\rm Span}_B\{v_1,...,v_n\} = L^2(TM)\restr{B}$.
If $L^2(TM)$ has a basis of cardinality $n$ on $B$, we say that it has dimension $n$ on $B$, or say that its local dimension on $B$ is $n$.
If $L^2(TM)$ does not admit any local basis of finite cardinality on  any subset of $B$ with positive measure, we say that the local dimension of $L^2(TM)$  on $B$ is infinity.
\end{definition}

It is proved in Proposition 1.4.4 \cite{G-N} that the  basis and dimension are well defined.
It can also be proven that there exits a unique decomposition $\{ E_n\}_{n \in \mathbb{N} \cup \{\infty\}}$ of $X$, such that for each $E_n$ with positive measure, $n \in \mathbb{N} \cup \{\infty\}$, $L^2(TM)$ has dimension $n$ on $E_n$. Furthermore, thanks to the infinitesimal Hilbertianity we have  the following proposition.

\begin{proposition}[Theorem 1.4.11, \cite{G-N}]\label{decomposition}
Let $\ms$ be a $\rcd$ metric measure space. Then there exists  a unique decomposition $\{ E_n\}_{n \in \mathbb{N} \cup \{\infty\}}$ of $X$ such that
\begin{itemize}
\item [a)] For any $n \in \mathbb{N}$ and any $B \subset E_n$ with finite positive measure,  $L^2(TM)$ has a unit orthogonal basis $\{e_{i,n}\}_{i=1}^n$ on $B$,
\item [b)] For every subset $B$ of $E_\infty$ with finite positive measure, there exists a unit orthogonal set $\{e_{i,B}\}_{i \in \mathbb{N} \cup \{\infty\}} \subset L^2(TM)\restr{B}$  which generates $L^2(TM)\restr{B}$,
\end{itemize}
where unit orthogonal of a countable set $\{v_i\}_i\subset L^2(TM)$ on $B$ means $\la v_i, v_j \ra=\delta_{ij}$ $\mm$-a.e. on $B$.
\end{proposition}

\begin{definition}[Analytic Dimension]
 Let $\{ E_n\}_{n \in \mathbb{N} \cup \{\infty\}}$ be the decomposition given in Proposition \ref{decomposition}. We define the local dimension $\dim_{\rm loc}:M \mapsto \mathbb{N}\cup \{\infty\}$ by $\dim_{\rm loc}(x)=n$ on $E_n$. We say that the dimension of $L^2(TM)$ is $k$ if $k=\sup \{n: \mm(E_n)>0\}$.  We define the analytic dimension of $M$ as the dimension of $L^2(TM)$ and denote it by $\dim_{\rm max} M$.
\end{definition}

We have the following proposition concerning the analytic dimension of $\rcdkn$ spaces.

\begin{proposition}[See \cite{H-R}]\label{finite dim}
Let $M=\ms$ be a $\rcdkn$ metric measure space. Then ${\dim}_{\rm max} M \leq N$. Furthermore, if the local dimension on a Borel set $E$ is $N$, we have
$\tr \H_f(x)=\Delta f (x)$ $\mm$-a.e. $x \in E$ for every $f \in {\rm TestF}$.
\end{proposition}

Combining the results in Proposition \ref{decomposition} and Proposition \ref{finite dim}, we know there is a coordinate system on $\rcdkn$ space $\ms$, i.e.  there exists a  partition of $X$: $\{E_n\}_{n\leq N}$,  such that $\dim_{\rm loc}(x)=n$ on $E_n$ and $\{e_{i,n}\}_i, n=1,..., \lfloor N \rfloor$ are the unit orthogonal  basis on corresponding $E_n$.  Then we can do computations on  $\rcdkn$ spaces in a similar way as on manifolds. For example, the pointwise Hilbert-Schmidt norm $|S |_{\rm HS}$  of a $L^\infty$-bilinear map $S:[L^2(TM)]^2 \mapsto L^0(M)$ could be defined in the following  way.  Let $S_1, S_2 :[L^2(TM)]^2 \mapsto L^0(M)$ be symmetric bilinear maps. We define $\la S_1, S_2 \ra_{\rm HS}$ as a function such that $\la S_1, S_2 \ra_{\rm HS}:=\sum_{i,j}S_1(e_{i,n}, e_{j,n})S_2(e_{i,n}, e_{j,n})$, $\mm$-a.e. on $E_n$. Clearly, this definition is well posed.  In particular, we can define the Hilbert-Schmidt norm of $S$ as $\sqrt{\la S, S \ra_{\rm HS} }$ and denote it by $|S|_{\rm HS}$.  It can be seen that it is compatible with the canonical definition of Hilbert-Schmidt norm on vector space. The trace of $S$ can be computed by $\tr S= \la S, \Id_{\dim_{\rm loc}} \ra_{\rm HS}$ where  $\Id_{\dim_{\rm loc}}$ is the unique  map satisfying $\Id_{\dim_{\rm loc}}(e_{i,n},
e_{j,n})=\delta_{ij}$, $\mm$-a.e. on $E_n$.

%%%%%%%%%%%%%%%%%%%%%%%%%%%%%%%%%%%%%%%%%%%%%%%%%%%%%
% å±å½¢åæ¢-åºæ¬å¼ç
%%%%%%%%%%%%%%%%%%%%%%%%%%%%%%%%%%%%%%%%%%%%%%%%%%%
\section{Main results}
\subsection{Conformal transformation}
In this section we  study the conformal transformation on metric measure space. Firstly we introduce  some basic definitions and facts. It can be seen that all the objects about  conformal transformation are intrinsically defined.

\bigskip

Let $w, v$ be  continuous and bounded functions on the metric measure space $M=\ms$. We construct a metric measure space $M':=(X, \d',\mm')$  by conformal transformation in the following way.
\begin{itemize}
\item [(1)] We replace $\mm$ by the weighted measure with density $e^{v}$:
\[
\mm'=e^{v}\,\mm.
\]
\item [(2)] We replace $\d$ by $\d'$,  the weighted metric with conformal factor $e^{w}$:
\[
\d'(x, y)=\mathop{\inf}_{\gamma} \Big\{ \int_0^1 |\dot{\gamma_t}|e^{w(\gamma_t)}\,\d t: \gamma \in {\rm AC}([0,1],X), \gamma_0=x, \gamma_1=y\Big\},
\]
where  ${\rm AC}([0,1],X)$ is the space of absolutely continuous curves on $X$. For simplicity, we  adopt the notation $\d'=e^{w}\d$ to emphasize the conformal factor.
\end{itemize}
\bigskip
Then we have a basic proposition. We say that a geodesic space is proper if  closed geodesic balls are compact.

\begin{proposition}\label{prop:geod}
Let $(X, \d)$ be a proper  geodesic space, $w$ be  a continuous and bounded function. Then the space $(X, \d')$ is  a geodesic space whose topology coincides with $(X,\d)$.
\end{proposition}
\begin{proof}
Since $w$ is continuous and bounded,  the weight  $e^{w}$ is also bounded and continuous. We know the topology of  $(X, \d')$  coincides with the topology of $(X, \d)$, and $(X, \d')$ is complete. Let $x, y\in X$ be two different points.  Assume that  $\{\gamma^n\}_n \subset {\rm AC}([0,1],X)$ is a minimizing sequence such that $\int_0^1 |\dot{\gamma_t}^n|\,\d t \to \d'(x,y)$.   Since the space $(X, \d)$ is proper,  and the mid-points of  $\gamma^n$ in $(X, \d')$ are bounded, we know that there exists a point $\gamma_{\frac12}\in X$ such that  $\d'(x, \gamma_{\frac12})=\d'(y, \gamma_{\frac12})=\frac12 \d'(x,y)$. Then we can build a geodesic $(\gamma_t)_t$ in $(X, \d')$ by repeating this `mid-point argument'. 
\end{proof}

It can be seen that the conformal transformation is reversible. The space $M$ can be obtained from $M'$ through conformal transformation, in which case the conformal factors $e^w, e^v$ should be replaced by $e^{-w}$ and $e^{-v}$.

\bigskip

Let $M :=\ms$ be a Riemannian manifold,   $w,v \in C^\infty$.   The gradient and Laplacian on $M':=(X, \d', e^v\mm)$  are denoted by $\nabla'$ and $\Delta'$ respectively. We have the following assertions.
\begin{itemize}
\item [1)] The Sobolev spaces $W^{1,2}(M)$ and $W^{1,2}(M')$ coincide as sets.
\item [2)] For any $f\in W^{1,2}(M)=W^{1,2}(M')$, we have $|\d f|_{M'}=e^{-w}|\d f|_M$,  and $\nabla' f=e^{-2w}\nabla f$.
\item [3)] For any $X, Y \in TM=TM'$, we have $\la X, Y \ra_{M'}=e^{2w} \la X, Y \ra$.
\item [4)] For any $u\in C^\infty(M)=C^\infty(M')$, we have ${\Delta}'u=e^{-2w}\big{(}{\Delta}u+\Gamma(v-2w,u)\big{)}$.
\end{itemize}

Next, we will prove the non-smooth counterparts of these properties. First of all,   we have a  simple lemma concerning the identification of Sobolev spaces.

\begin{lemma}[Lemma 3.11, \cite{GH-W}] \label{gradient-compare}
Let $M=(X,\d, \mm)$, $M'=(X,\d', \mm')$ be metric measure spaces. Assume that $\d\geq\d'$, $\d'$ induces the same topology as $\d$,  and $c\mm\leq \mm' \leq C\mm$ for some $c, C>0$. Then $W^{1,2}(M') \subset W^{1,2}(M)$, and for any function $f \in W^{1,2}(M')$, we have: $|\D f|_M \leq |\D f|_{M'}$ $\mm$-a.e..
\end{lemma}

\begin{proposition}[Sobolev space under conformal transformation]\label{lemma-1}
Let $M=\ms$ be a metric measure space, $v,w$ be   bounded continuous functions.  The space $M':=(X, \d', \mm')$  is  constructed by conformal transformation as above. Then $W^{1,2}(M)$ and $W^{1,2}(M')$ coincide as sets and
\[
|\D f|_{M'}=e^{-w}|\D f|_M,~~~\mmae
\]
for any $f\in W^{1,2}(M)$. In particular, if $M$ is infinitesimally Hilbertian,  $M'$ is also infinitesimally Hilbertian.
\end{proposition}
\begin{proof}
Let $\epsilon >0$, $x \in X$. Pick $r>0$ such that
\[
 \mathop {\sup}_{y \in B_r(x)}\max \Big{\{}\frac{e^{w(x)}}{e^{w(y)}}, \frac{e^{w(y)}}{e^{w(x)}}\Big{\}}<1+\epsilon,
\]
where $B_r(x)$ is the closed ball in $(X, \d)$ with radius $r$.

Then for any Lipschitz function $g$, we have
\begin{eqnarray*}
\loclip{g}{M'}(x)&=&\lmts{y}{x} \frac{|g(y)-g(x)|}{\d'(y,x)}=\lmts{ B_r(x) \ni y}{x} \frac{|g(y)-g(x)|}{\d'(y,x)}\\
&\leq& (1+\epsilon)e^{-w(x)} \lmts{ B_r(x) \ni y}{x} \frac{|g(y)-g(x)|}{\d(y,x)}\\
&=& (1+\epsilon)e^{-w(x)} \loclip{g}{M}(x).
\end{eqnarray*}
Similarly, we have $\loclip{g}{M'}(x) \geq (1+\epsilon)^{-1}e^{-w(x)} \loclip{g}{M}(x)$. Since the choice of $\epsilon$ is arbitrary, we know
\[
\loclip{g}{M'}(x) =e^{-w(x)} \loclip{g}{M}(x).
\]

Since $e^{v}$ is  bounded and continuous,  we know that $L^2(M)$ coincides with $L^2(M')$.
For any $f\in W^{1,2}(M)$, from Proposition \ref{prop-density}  we know there exits a sequence of Lipschitz functions $\{ f_n\}_n$  such that $f_n \rightarrow f$ and $\loclip{f_n}{M} \rightarrow |\D f|_M$ in $L^2(X, \mm)$. Then we know $f_n \rightarrow f$ and $\loclip{f_n}{M'} \rightarrow e^{-w}|\D f|_M$ in $L^2(X, \mm')$.
Thus from lower semi-continuity we know $|\D f|_{M'} \leq e^{-w}|\D f|_M$, $\mm$-a.e.. Hence $f\in W^{1,2}(M')$ and we have $W^{1,2}(M) \subset W^{1,2}(M')$.

Conversely, we can exchange the roles of $M$ and $M'$, i.e. $M$ can be obtained from $M'$ through conformal transformation, with  the conformal factors  $e^{-w}, e^{-v}$. Using the same argument we can prove $|\D f|_{M'} \geq e^{-w}|\D f|_M$, $\mm$-a.e., and  $W^{1,2}(M') \subset W^{1,2}(M)$,  then we complete the proof.

\end{proof}

\bigskip
\begin{remark}
  General weighted Sobolev space is studied in \cite{APS-W}. Thus we can also characterize the Sobolev space under  conformal transformation for unbounded $v$.  In case $w$ is unbounded, the problem is more complicated. Here we introduce  a possible approach, the idea comes from \cite{APS-W}.

Let $w$ be a continuous function with $e^{2w}\in L^{1}(\mm) \cap L^{-1}_{\rm loc}(\mm)$. We define the weighted Sobolev space $W^{1,2}_w\ms$ by
\[
W^{1,2}_w:=\Big\{ f\in W^{1, 1}\ms: \int |f|^2\,\d\mm+\int |\D f|^2\,e^{-2w}\d\mm<\infty\Big \},
\]
where the definition of $W^{1,1}$ can be found in \cite{Ambrosio-DiMarino14} and \cite{GH-I}.
We endow $W^{1,2}_w\ms$ with the norm:
\[
\|f\|^2_w:=\int |f|^2\,\d\mm+\int |\D f|^2\,e^{-2w}\d\mm.
\]
Using H\"older inequality we can prove that  $W^{1,2}_w$ embeds continuously into $W^{1,1}$. We  define 
$H^{1,2}_w\ms$ as the closure of Lipschitz functions with compact support in $W^{1,2}_w$. 

We may prove  $W^{1,2}_w=W^{1,2}(X, e^w\d, \mm)$ and $H^{1,2}_w=W^{1,2}(X, e^w\d, \mm)$ under further assumption. This is an independent topic, which will be studied in a forthcoming paper.
\end{remark}

\bigskip

%%%%%%%%%%%%%%%%%%%%%%%%%%%%%%%%%%%%%%%%%%%%%%%%%%%%%%%%%%%%%%%%%%%%%%%%%%%%%%%%%%%%%%%%%%%%%%%%%%%%%

  Let $M$ be an infinitesimally Hilbertian space. From the results above, we know $\Gamma'(\cdot, \cdot)=e^{-2w}\Gamma( \cdot, \cdot)$.  The  energy form on $M'$ is defined by
\[
W^{1,2}(M') \ni f \mapsto \int \Gamma'(f)\,\d\mm'=\int |\D f|^2\,e^{v-2w}\,\d\mm.
\]

The Laplacian on $M'$ can be represented in the following way.
\begin{proposition}[Laplacian under conformal transformation] \label{lemma-2}
Assume that $M$ and $M'$ are infinitesimally Hilbertian metric measure spaces. Assume $v,w \in W^{1,2}(M)\cap L^\infty(M)$. Then ${\rm D} ({\bf \Delta}')={\rm D} ({\bf \Delta})$
and ${\rm TestF}(M')={\rm TestF}(M)$.  For any $f \in {\rm D}({\bf \Delta}')$, we have
 \[
{\bf \Delta}'f=e^{v-2w}\big{(}{\bf \Delta}f+\Gamma(v-2w,f)\,\mm\big{)}.
\]

Furthermore, we have
the formula:
\[
{\Delta}'f=e^{-2w}\big({\Delta}f+\Gamma(v-2w,f)\big ),~~~\mmae
\]
where $f \in {\rm TestF}(M')$.
\end{proposition}
\begin{proof}
Let $f \in {\rm D} ({\bf \Delta})$. By definition, there exists a measure ${\bf \Delta} f$ such that
\[
\int \phi \, \d {\bf \Delta} f=-\int \Gamma(\phi, f)\,\d\mm
\]
for any Lipschitz function $\phi$ with bounded support.

Thus for any Lipschitz function $\varphi$ with bounded support, we have
\begin{eqnarray*}
&&\int \varphi e^{v-2w}\big{(}\d{\bf \Delta}f+\Gamma(v-2w,f)\,\d\mm\big{)}\\ &=&\int e^{v-2w}\varphi\, \d{\bf \Delta}f+\int \varphi \Gamma(v-2w,f)\,e^{v-2w}\,\d\mm \\
&=& -\int \Gamma(e^{v-2w}\varphi, f)\,\d\mm+\int \varphi \Gamma(v-2w,f)\,e^{v-2w}\,\d\mm\\
&=& -\int e^{v-2w}\Gamma(\varphi, f)\,\d\mm+\int \varphi e^{v-2w}\Gamma(2w-v, f)\,\d\mm+\int \varphi \Gamma(v-2w,f)\,e^{v-2w}\,\d\mm\\
&=& -\int e^{v-2w}\Gamma(\varphi, f)\,\d\mm\\
&=& -\int \Gamma'(\varphi, f)\,\d\mm'.
\end{eqnarray*}
Therefore we know $f\in {\rm D} ({\bf \Delta}')$, and by uniqueness we have
 \[
{\bf \Delta}'f=e^{v-2w}({\bf \Delta}f+\Gamma(v-2w,f)\,\mm).
\]

Conversely, we can prove ${\rm D} ({\bf \Delta}')\subset {\rm D} ({\bf \Delta})$. Combining with Proposition \ref{lemma-1} we get ${\rm TestF}(M')={\rm TestF}(M)$. When ${\bf \Delta}'f \ll \mm'$, we know ${\bf \Delta}f \ll \mm$ and
\[
{\Delta}'f=e^{-2w}({\Delta}f+\Gamma(v-2w,f)),~~~\mmae.
\]
\end{proof}
\bigskip
\begin{proposition}[Gradient under conformally transformation]  \label{lemma-3}
Let $M$ and $M'$ be  metric measure spaces as discussed above. Then
$\nabla' f=e^{-2w}\nabla f$
and $\la X, Y \ra_{M'}=e^{2w} \la X, Y \ra$, $\mm$-a.e.
\end{proposition}
\begin{proof}
Let $f, g\in W^{1,2}(M)=W^{1,2}(M')$. From definition we know
\[
\d f(\nabla' g)=\la \nabla' f, \nabla' g \ra_{M'}
\]
and
\[
\la \nabla' f, \nabla' g \ra_{M'}=\Gamma'(f,g)=e^{-2w}\Gamma(f,g)=e^{-2w}\la \nabla f, \nabla g \ra_{M}=e^{-2w}\d f(\nabla g).
\]
Then we know $\d f(\nabla' g)=e^{-2w}\d f(\nabla g)$, therefore, $\nabla' g=e^{-2w} \nabla g$.

Furthermore, we have
\[
\la X, Y \ra_{M'}=e^{2w}\la X, Y \ra_{M},
\]
for any $X=\nabla f, Y=\nabla g$, $f, g\in W^{1,2}(M)$. By linearity and the density we know $\la X, Y \ra_{M'}=e^{2w}\la X, Y \ra_{M}$ holds
for any $X, Y \in L^2(TM)$.
\end{proof}

We then have the following  corollary.

\begin{corollary}[Conformal transformation preserves angle]\label{prop:angle}
Let $X, Y \in L^2(TM)$. The conformal transformation preserves the angle between $X$ and $Y$ which is  defined by
\[
\angle XY (x):=\arccos \frac{\la X, Y \ra}{|X||Y|}(x).
\]
\end{corollary}

The following property is a crucial condition for non-smooth Bakry-\'Emery theory (see \cite{AGS-B}, \cite{S-S}, \cite{GH-W}).
\bigskip
\begin{proposition}[Conformal transformation preserves Sobolev-to-Lipschitz property] \label{lemma-4}
Let $M$ be a $\rcdkn$ metric measure space with $N <\infty$, and $M'$ be constructed by conformal transformation  as in Proposition \ref{lemma-1}. Then $M'$  satisfies the Sobolev-to-Lipschitz property.
\end{proposition}
\begin{proof}
From Bishop-Gromov inequality we know that $\rcdkn$ spaces are proper. So  we know  $M'$ is  geodesic from Proposition \ref{prop:geod}. Let $f \in W^{1,2}(M')$  with $|\D f|_{M'} \in L^\infty(M')$. By Proposition \ref{lemma-1}, we have $|\D f|_{M'}=e^{-w} |\D f| \in L^\infty(M')$. Since $M$ has Sobolev-to-Lipschitz property (see \cite{AGS-M}),  $f$ has a  Lipschitz representation $f'$ on  $M$,  and $\Lip_{M}(f')=\esup{} ~|\D f|$.  So  $f'$ is also Lipschitz on  $M'$.

We claim that 
$\loclip{f'}{M}(x)\leq \mathop{\esup{}}_{y\in B_r(x)} ~|\D f|(y)$ on any closed ball $B_r(x)$.  For any $\epsilon >0$ and $y\in B_{r-\epsilon}(x)$, we know $B_\epsilon(y) \subset B_r(x)$. Then we consider the optimal transport from $\frac1{\mm(B_\epsilon(y))}\mm\restr{B_\epsilon(y)}$ to $\delta_x$. We know that there exists a geodesic $(\mu_t)$ connecting them,  and there exists an optimal transport plan $\Pi\in \mathcal{P}({\rm Geod}(X, \d))$ such that $(e_t)_\sharp\Pi=\mu_t$. It is known (see \cite{R-IG, R-IM}) that  $\mu_t,  t\in [0, 1-\eta]$ have uniformly bounded densities for any $\eta>0$. By an equivalent characterization of Sobolev function (see \cite{AGS-C}) we have
\begin{eqnarray*}
\Big |\int f'\,\mu_{1-\eta}-\int f'\,\mu_0\Big |&=& \Big | \int \big(f'(\gamma_{1-\eta})-f'(\gamma_0)\big )\,\d\Pi(\gamma) \Big | \\
&\leq&  \int \int_0^{1-\eta} |\D f|(\gamma_t)|\dot{\gamma}_t|\,\d t \d \Pi(\gamma)      \\
&\leq&   \Big ( \int  \int_0^{1-\eta} |\D f|^2(\gamma_t)\,\d t \d \Pi(\gamma)   \Big)^\frac12 \Big ( \int  \int_0^{1-\eta} |\dot{\gamma}_t|^2\,\d t \d \Pi(\gamma)   \Big)^\frac12\\
&=&  \Big ( \int_0^{1-\eta}\int   |\D f|^2\, \d \mu_t\d t  \Big)^\frac12 \Big (  \int (1-\eta)\d^2(\gamma_1, \gamma_0)\,\d \Pi(\gamma)   \Big)^\frac12 \\
&\leq & (1-\eta)\big ( \mathop{\esup{}}_{y\in B_r(x)} ~|\D f|(y)\big ) (\d(x,y)+\epsilon).
\end{eqnarray*}
Letting $\eta \to 0$,  and $\epsilon \to 0$, we have
\[
\Big | \frac {f'(x)-f'(y)}{\d(x,y)}\Big | \leq  \mathop{\esup{}}_{y\in B_r(x)} ~|\D f|(y).
\]
Taking $y\to x$, we get 
\begin{equation}\label{eq1:slp}
\loclip{f'}{M}(x)\leq \mathop{\esup{}}_{y\in B_r(x)} ~|\D f|(y).
 \end{equation}

 We need  to prove
$\Lip_{M'} (f')=\esup ~{|\D f|_{M'}}=\esup ~{e^{-w} |\D f|}$. As the inequality $\Lip_{M'} (f') \geq \esup ~{|\D f|_{M'}}$ is trivial, we just need to prove the opposite one.

From the proof of Proposition \ref{lemma-1} we know
\[
\loclip{f'}{M'}(x) =e^{-w(x)} \loclip{f'}{M}(x).
\]

Combining with \eqref{eq1:slp}, we have 
\begin{eqnarray*}
\loclip{f'}{M'}(x) &=& e^{-w(x)}\loclip{f'}{M}(x)\\
&\leq&  e^{-w(x)} \lmt{r}{0}\mathop{\esup{}}_{y \in B_r(x)} |\D f|(y) \\
&=& \lmt{r}{0}\mathop{\esup{}}_{y \in B_r(x)} e^{-w(y)}|\D f|(y) \\
&\leq& \esup ~{ e^{-w(y)}|\D f|(y) }\\
&=& \esup~{ |\D f|_{M'}}.
\end{eqnarray*}
 
 Since $ M'$ is a geodesic metric  space, we know 
$\Lip_{M'} (f')=\sup_x \loclip{f'}{M'}(x)$, hence $\Lip_{M'} (f') \leq \esup ~{|\D f|_{M'}}$.  
In conclusion, we obtain  $\Lip_{M'} (f') = \esup ~{|\D f|_{M'}}$ and we complete the proof.
\end{proof}

%%%%%%%%%%%%%%%%%%%%%%%%%%%%%%%%%%%%%%%%%%%%%%%%%%%%%%%%%%%%%%%%%%
%% ä¸»è¦ç»æ
%%%%%%%%%%%%%%%%%%%%%%%%%%%%%%%%%%%%%%%%%%%%%%%%%%%%%%%%%%%%%%%%

\subsection{Ricci curvature tensor under conformal transformation}
In this section we study the Ricci tensor under conformal transformation, and  prove the transformation formula  in Theorem \ref{th-conformal}.

First of all, from Proposition \ref{finite dim}  we know that (see  \cite{H-R}, and \cite{G-N} for the case $N=\infty$) the following $N$-Ricci tensor $\bRicn$ is well-defined.
\begin{definition}[$N$-Ricci tensor] \label{Ricn} For any $f \in {\rm TestF}(M)$, we define the $N$-Ricci tensor  $\bRicn(\nabla f, \nabla f) \in {\rm Meas}(M)$  by
\begin{eqnarray*}
\bRicn(\nabla f, \nabla f):={\bf \Gamma}_2(f) -\Big{(} |\H_f|^2_{\rm HS}+\frac1{N-{\dim_{\rm loc}}}(\tr \H_f -\Delta f)^2 \Big{)} \, \mm.
\end{eqnarray*}
\end{definition}
\bigskip
We  recall the following result which has been proven in  Theorem 4.4, in \cite{H-R}. Here we modify the statement of the theorem according to Proposition \ref{becondition}. 

\begin{theorem} \label{th-ricci}
Let $M$ be a $\rcdkn$ space. Then
\[
\bRicn(\nabla f, \nabla f) \geq K|\D f|^2\,\mm
\]
 and
\begin{eqnarray}\label{ineq-main}
{\bf \Gamma}_2(f) &\geq & \Big{(} \frac{(\Delta f)^2}{N}\Big{)}\,\mm+\bRicn(\nabla f, \nabla f) 
\end{eqnarray}
hold for any $ f \in {\rm TestF}(M)$.
Conversely, let $M$ be  an  infinitesimal Hilbert space satisfying Assumption \ref{assumption} and the Sobolev-to-Lipschitz property, ${\rm TestF}(M)$ be dense in $W^{1,2}(M)$. Assume that
\begin{itemize}
\item [(1)] ${\rm dim}_{\rm max} M \leq N$,
\item [(2)] $\tr \H_f=\Delta f~ \mmae \text{ on}~~ \{{\dim}_{\rm loc}=N\}, \forall f \in {\rm TestF}(M)$,
\item [(3)]  $\bRicn(\nabla f, \nabla f)  \geq K|\D f|
^2\,\mm$,    $\forall f \in {\rm TestF}(M)$,
\end{itemize}
for some $K \in \R$, $N \in [1,+\infty]$, then $M$ is  $\rcdkn$.
\end{theorem}

\bigskip
According to the Definition \ref{Ricn}, we need to compute the Hilbert-Schmidt norm of the Hessian under conformal transformation. We have the
following lemma. Notice that the factor $e^v$ has nothing to do with the Hessian.

\begin{proposition}[Hessian under conformal transformation]\label{Hessian}
Let $M=\ms$ be a $\rcdkn$ metric measure space, $e^{w}$ be a conformal factor with $w\in {\rm TestF}(M)$. Then for any $f \in {\rm TestF}(M)$, the following formulas
\begin{eqnarray*}
|\H'_f|^2_{\rm HS}=e^{-4w}\big{(}|\H_f|^2_{\rm HS}&+&2\Gamma(f)\Gamma(w)+({\dim_{\rm loc}}-2) \Gamma(f,w)^2\\
&-&2\Gamma(w,\Gamma(f))+2\Gamma(f,w)\tr \H_f \big{)}
\end{eqnarray*}
and
\[
\tr \H'_f=e^{-2w} \big{(} \tr \H_f+({\dim_{\rm loc}}-2) \Gamma(f,w) \big{)}
\]
hold $\mm$-a.e. .
\end{proposition}
\begin{proof}
Let $g, h$ be arbitrary test functions. By direct computation, we  have
\begin{eqnarray*}
&&\H'_f(\nabla' g, \nabla' h)\\
&=&\frac12 \big{(}\Gamma'(g, \Gamma'(f,h))+\Gamma'(h, \Gamma'(f,g))-\Gamma'(f, \Gamma'(g,h)) \big{)}\\
&=&  \frac{e^{-2w}}2 \big{(}\Gamma(g, e^{-2w}\Gamma(f,h))+\Gamma(h, e^{-2w}\Gamma(f,g))-\Gamma(f, e^{-2w}\Gamma(g,h)) \big{)}\\
&=& \frac{e^{-4w}}2\big{(}\Gamma(g, \Gamma(f,h))+\Gamma(h, \Gamma(f,g))-\Gamma(f, \Gamma(g,h))\\
&&~~~~~~~~-2\Gamma(g,w)\Gamma(f,h)-2\Gamma(h,w)\Gamma(f,g)+2\Gamma(f,w)\Gamma(g,h)\big{)}\\
&=& e^{-4w}\big{(} \H_f(\nabla g, \nabla h)-\Gamma(g,w)\Gamma(f,h)-\Gamma(h,w)\Gamma(f,g)+\Gamma(f,w)\Gamma(g,h)\big{)}\\
&=& e^{-4w}\big{(} \H_f(\nabla g, \nabla h)-\la \nabla g, \nabla w \ra \la \nabla f, \nabla h \ra -\la \nabla h, \nabla w \ra \la \nabla f, \nabla g \ra \\
&&~~~~~~~~+\la \nabla f, \nabla w \ra \la \nabla g, \nabla h\ra \big{)}.
\end{eqnarray*}

Then we replace $g, h$ by linear combinations of test functions in the equalities above. On one hand,  we  replace $\nabla'g$ by $\sum_i \nabla'g_i$, and  $\nabla' h$ by $\sum_j \nabla' h_j$ in $\H'_f(\nabla' g, \nabla' h)$. Then by approximation and the continuity of Hessian as a bilinear map from $[L^2(TM)]^2$ to $L^1(M)$, we can replace $\nabla'g, \nabla' h$ by $e'_i, e'_j$ where $\{ e'_i \}_i$ is a unit orthogonal base on $M$ with respect to $\Gamma'(\cdot, \cdot)$. On the other hand,  from Lemma \ref{lemma-3} we know that $\nabla g$ and $\nabla h$ should be simultaneously  replaced by $e^w e_i$ and $e^w e_j$ where $\{e_i\}_i$ is the corresponding unit orthogonal base with respect to $\Gamma(\cdot, \cdot)$. Therefore we obtain
\[
(\H'_f)_{ij}=e^{-2w}\big{(}(\H_f)_{ij}-w_if_j-w_jf_i+\Gamma(f,w)\delta_{ij} \big{)},
\]
$\mm$-a.e., where we keep the notion $(T)_{ij}=T(e_i, e_j)$ for a bilinear map $T$ and $f_i=\la \nabla f, e_i \ra$ for a function $f$.
Then we have
\begin{eqnarray*}
|\H'_f|^2_{\rm HS}&=&\mathop{\sum}_{i,j}(\H'_f)^2_{ij}\\
&=& e^{-4w}\mathop{\sum}_{i,j}\big{(}(\H_f)_{ij}-w_if_j-w_jf_i+\Gamma(f,w)\delta_{ij} \big{)}^2\\
&=& e^{-4w}\big{(}|\H_f|^2_{\rm HS} +2\Gamma(f)\Gamma(w)+({\dim_{\rm loc}}-2) \Gamma(f,w)^2\\
&&~~~-2\Gamma(w,\Gamma(f))+2\Gamma(f,w)\tr \H_f \big{)},
\end{eqnarray*}
 $\mm$-a.e..

In the same way, we can prove
\[
\tr \H'_f(x)=\mathop{\sum}_{i=j} (\H'_f)_{ij}=e^{-2w} \big{(} \tr \H_f-2\Gamma(f,w)+{\dim_{\rm loc}}(x) \Gamma(f,w) \big{)}
\]
for $\mm$-a.e. $x\in X$.
\end{proof}
\bigskip
In \cite{G-N}, Gigli  defines the space $W^{1,2}_H(TM)$ which is the closure of test vectors with respect to an appropriate Sobolev norm.  For any vector $X \in W^{1,2}_H(TM)$,  the notion of `covariant derivative' $\nabla X$ is a well defined bounded bilinear map from $[L^2(TM)]^2$ to $L^1(M)$. It has been proven that  $\H_f=\nabla \nabla f$ for any test function $f$.  Since the map $X \to \nabla X$ is continuous in $W^{1,2}_H(TM)$ , by density and linearity  we  can extend the transformation formula for Hessian in the following corollary.

\begin{corollary}[Covariant derivative under conformal transformation]\label{Cov}
Let $X \in W^{1,2}_H(TM)$. Then 
\begin{eqnarray*}
\nabla' X:'( Y \otimes Z)
&=& \nabla X:( Y \otimes Z)-\la Y, \nabla w \ra \la X, Z \ra -\la Z, \nabla w \ra \la X, Y \ra \\
&&~~~~~~~~+\la X, \nabla w \ra \la Y, Z\ra 
\end{eqnarray*}
\end{corollary}
%%%%%%%%%%%%% Ricci-æ²ççåæ¢å¬å¼
\bigskip
\begin{theorem}[Ricci tensor under conformal transformation]\label{th-conformal}
Let $M=\ms$ be a $\rcdkn$ metric measure space,  $w,v\in {\rm TestF}(M)$. The  metric measure space constructed through conformal transformation is
$M'=(X,\d',\mm')$,  where $\d'=e^{w}\d$ and $\mm'=e^{v}\mm$. Then for any $N' \in \R$, the $N'$-Ricci tensor on $M'$ can be computed in the following way:
\begin{eqnarray*}
&&{\bf Ricci}'_{N'}(\nabla' f, \nabla' f)\\
&=&e^{v-4w} {\bf Ricci}_{N'}(\nabla f, \nabla f)\\
&&+e^{v-4w} \Big{(}\Gamma(w,f)\big(\frac{(2-\dim_{\rm loc})(N'-2)}{N'-\dim_{\rm loc}} \Gamma(w,f)+2\Delta f+2\Gamma(v-2w,f)\\
&&-2\tr \H_f-\frac{2(\Delta f-\tr \H_f)(2-\dim_{\rm loc})}{N'-\dim_{\rm loc}}\big )\\
&&-\H_{v-2w}(\nabla f,\nabla f)
-\Gamma(f) \big{(} \Gamma(v-2w, w)+\Delta w \big{)} \\
&&-\frac{\Gamma(v-2w,f)}{N'-\dim_{\rm loc}}\big( \Gamma(v-2w,f) +2(\Delta f-\tr \H_f)+(4-2\dim_{\rm loc})\Gamma(f,w)\big)
\Big{)}\,\mm.
\end{eqnarray*}
\end{theorem}
\begin{proof}
According to the Definition \ref{Ricn}, we firstly compute ${\bf \Gamma}'_2(f)$ for any $f \in {\rm TestF}(M)$. From definition we know
\[
{\bf \Gamma}'_2(f)=\frac12 {\bf \Delta}'(\Gamma'(f))-\Gamma'(\Delta' f, f)\,\mm'.
\]

By Proposition \ref{lemma-1} and Proposition \ref{lemma-2} we have
\begin{eqnarray*}
{\bf \Delta}'(\Gamma'(f))&=& e^{v-2w} \big{(} {\bf \Delta} (e^{-2w}\Gamma(f))+\Gamma(v-2w,e^{-2w}\Gamma(f))\,\mm  \big{)} \\
&=& e^{v-2w} \big{(} {\bf \Delta} (e^{-2w})\Gamma(f)+ 2\Gamma(e^{-2w}, \Gamma(f))\,\mm +e^{-2w} {\bf \Delta}(\Gamma(f)) \\
&&~~~~+e^{-2w}\Gamma(v-2w,\Gamma(f))\,\mm-2e^{-2w}\Gamma(f)\Gamma(v-2w, w)\,\mm \big{)} \\
&=& e^{-4w} \big{(} 4\Gamma(f)\Gamma(w)-2\Delta w \Gamma(f)-4\Gamma(w, \Gamma(f))\\
&&~~~~~+\Gamma(v-2w, \Gamma(f))-2\Gamma(f)\Gamma(v-2w, w) \big{)}\,\mm' +e^{v-4w} {\bf \Delta}(\Gamma(f)),\\
\end{eqnarray*}
and
\begin{eqnarray*}
\Gamma'(\Delta' f, f)&=& e^{-2w} \big{(} \Gamma(f, e^{-2w}(\Delta f + \Gamma(v-2w,f)) \big{)} \\
&=& e^{-4w}  \big{(} \Gamma(f, \Delta f +\Gamma(v-2w,f))  -2\Gamma(w, f)(\Delta f +\Gamma(v-2w,f))\big{)}\\
&=& e^{-4w}  \big{(} \Gamma(f, \Delta f) -2\Gamma(w, f)\Delta f -2\Gamma(w,f)\Gamma(v-2w,f) \\
&&~~~~~~~~~~+\Gamma(f, \Gamma (f, v-2w))\big{)}.
\end{eqnarray*}
Then  we have
\begin{eqnarray*}
{\bf \Gamma}'_2(f)&=& e^{v-4w} \big{(} {\bf \Gamma}_2(f) \big{)} +e^{-4w}\big{(}  2\Gamma(f)\Gamma(w)-\Delta w \Gamma(f) -2\Gamma(w,\Gamma(f))\nonumber \\ 
&&~~~~-\Gamma(f)\Gamma(v-2w, w)+\frac12 \Gamma(v-2w, \Gamma(f)) \\
&&~~~~-\Gamma(f, \Gamma (f, v-2w))+2\Gamma(f,w)\Delta f+2\Gamma(w,f)\Gamma(v-2w,f)  \big{)}\,\mm'.~~~ 
\end{eqnarray*}
By Definition \ref{Ricn}, Proposition \ref{Hessian} and the formula  above we have
\begin{eqnarray*}
&&{\bf Ricci}'_{N'}(\nabla' f, \nabla' f)\\
&:=&{\bf \Gamma}'_2(f) -|\H'_f|^2_{\rm HS}-\frac1{N'-{\dim_{\rm loc}}}(\tr \H'_f -\Delta' f)^2 \big{)} \, \mm'\\
&=& e^{v-4w} \big{(} {\bf \Gamma}_2(f) \big{)} +e^{-4w}\big{(}  2\Gamma(f)\Gamma(w)-\Delta w \Gamma(f) -2\Gamma(w,\Gamma(f))\nonumber \\ 
&&~-\Gamma(f)\Gamma(v-2w, w)+\frac12 \Gamma(v-2w, \Gamma(f)) \\
&&~-\Gamma(f, \Gamma (f, v-2w))+2\Gamma(f,w)\Delta f+2\Gamma(w,f)\Gamma(v-2w,f)  \big{)}\,\mm'\\
&&-e^{-4w}\big{(}|\H_f|^2_{\rm HS}+2\Gamma(f)\Gamma(w)+({\dim_{\rm loc}}-2) \Gamma(f,w)^2\\
&&~ -2\Gamma(w,\Gamma(f))+2\Gamma(f,w)\tr \H_f \big{)}\,\mm'\\
&&-\frac{e^{-4w}}{N'-\dim_{\rm loc}}\big{(} \Delta f-\tr \H_f+(2-\dim_{\rm loc}) \Gamma(f,w)+\Gamma(v-2w, f)\big{)}^2\,\mm'\\
&=& e^{v-4w} \big{(} {\bf \Gamma}_2(f)- |\H_f|^2_{\rm HS}\,\mm -\frac1{N'-\dim_{\rm loc}}( \Delta f-\tr \H_f)^2\,\mm \big{)}\\
&&~+e^{-4w} \Big{(}\Gamma(w,f)\big(\frac{(2-\dim_{\rm loc})(N'-2)}{N'-\dim_{\rm loc}} \Gamma(w,f)+2\Delta f+2\Gamma(v-2w,f)\\
&&~-2\tr \H_f-\frac{2(\Delta f-\tr \H_f)(2-\dim_{\rm loc})}{N'-\dim_{\rm loc}}\big )\\
&&~-\H_{v-2w}(\nabla f,\nabla f)
-\Gamma(f) \big{(} \Gamma(v-2w, w)+\Delta w \big{)} \\
&&~-\frac{\Gamma(v-2w,f)}{N'-\dim_{\rm loc}}\big( \Gamma(v-2w,f) +2(\Delta f-\tr \H_f)+(4-2\dim_{\rm loc})\Gamma(f,w)\big)
\Big{)}\,\mm' \\
&=&e^{v-4w} {\bf Ricci}_{N'}(\nabla f, \nabla f)\\
&&~+e^{-4w} \Big{(}\Gamma(w,f)\big(\frac{(2-\dim_{\rm loc})(N'-2)}{N'-\dim_{\rm loc}} \Gamma(w,f)+2\Delta f+2\Gamma(v-2w,f)\\
&&~-2\tr \H_f-\frac{2(\Delta f-\tr \H_f)(2-\dim_{\rm loc})}{N'-\dim_{\rm loc}}\big )\\
&&~-\H_{v-2w}(\nabla f,\nabla f)
-\Gamma(f) \big{(} \Gamma(v-2w, w)+\Delta w \big{)} \\
&&~-\frac{\Gamma(v-2w,f)}{N'-\dim_{\rm loc}}\big( \Gamma(v-2w,f) +2(\Delta f-\tr \H_f)+(4-2\dim_{\rm loc})\Gamma(f,w)\big)
\Big{)}\,\mm' 
\end{eqnarray*}
which is the result we need.
\end{proof}

\bigskip
As a corollary, we have the non-smooth version of the formula \eqref{eq:conformal}.

\begin{corollary}\label{coro-N}
Let $M=\ms$ be a $\rcdkn$ metric measure space, $e^{w}$ be the conformal factor with $w\in {\rm TestF}(M)$. The corresponding metric measure space under conformal transformation is
$M'=(X,\d',\mm')$ where $\d'=e^{w}\d$ and $\mm'=e^{Nw}\mm$. Then the $N$-Ricci tensor of $M'$ satisfies the following formula:
\begin{eqnarray*}
{\bf Ricci}'_N(\nabla' f, \nabla' f)=e^{(N-4)w}\big{(} \bRicn(\nabla f, \nabla f)&+&[-{\Delta}w-(N-2)\Gamma(w)]\Gamma(f)\,\mm \\
&-&(N-2)[\H_w(\nabla f, \nabla f)-\Gamma(w,f)^2]\,\mm \big{)}.
\end{eqnarray*}
\end{corollary}
\bigskip
We  end this article with two corollaries  concerning  the curvature-dimension condition under comformal transformation.  These results have been  proven in \cite{S-R} for smooth metric measure spaces.

\begin{corollary}\label{coro-conformal}
Let $M$ be a $\rcdkn$ space,  $M'=(X,\d',\mm')$ where $\d'=e^{w}\d$ and $\mm'=e^{Nw}\mm$,  $w \in {\rm TestF}$.  Then $M'$ satisfies  ${\rm RCD}^*(K', N)$ condition if
\begin{eqnarray*}
K':=\mathop{\inf}_{x\in X} e^{-2w}\Big{[} K-{\Delta}w-(N-2)\Gamma(w)
-\mathop{\sup}_{f \in {\rm TestF}(M)}\frac{N-2}{\Gamma(f)}\big{(}\H_w(\nabla f, \nabla f)-\Gamma(w,f)^2\big{)}\Big{]}
\end{eqnarray*}
is a real number.
\end{corollary}
\begin{proof}
 We know that $M'$ is infinitesimally Hilbertian from Lemma \ref{lemma-1}, $M'$ has Sobolev-to-Lipschitz property from Lemma \ref{lemma-4} and ${\rm TestF}(M')$ is dense in $W^{1,2}(M')$ from Lemma \ref{lemma-2}.   It is sufficient to check the conditions (1),(2) in the Theorem \ref{th-ricci}.

(1) By definition and Lemma \ref{lemma-1} we know that the conformal transformation will not change the local/analytic dimension. Hence by Proposition \ref{finite dim} we know ${\rm dim}_{\rm max} M' \leq N$.

(2) Let $f \in {\rm TestF}(M)={\rm TestF}(M')$. It is proved in Lemma \ref{lemma-2} and  Lemma \ref{Hessian} that
\[
\tr \H'_f-\Delta' f=e^{-2w} \big{(} \tr \H_f-{\Delta}f+({\dim_{\rm loc}}-N) \Gamma(f,w) \big{)}.
\]
On the set $\{{\dim}_{\rm loc}M=N\}=\{{\dim}_{\rm loc}M'=N\}$, we know $\tr \H_f={\Delta}f$ by Proposition \ref{finite dim}. Therefore, $\tr \H'_f=\Delta' f$
$\mm$-a.e. on $\{{\dim}_{\rm loc}M'=N\}$.

We can see that 
\begin{eqnarray*}
{\bf Ricci}'_N(\nabla' f, \nabla' f)&=&e^{(N-4)w}\big{(} \bRicn(\nabla f, \nabla f)+[-{\Delta}w-(N-2)\Gamma(w)]\Gamma(f)\,\mm \\
&&-(N-2)[\H_w(\nabla f, \nabla f)-\Gamma(w,f)^2]\,\mm \big{)}\\
 &\geq&K'\Gamma'(f)\,\mm'= K'e^{(N-2)w}|\D f|^2\,\mm
\end{eqnarray*}
 if 
\begin{eqnarray*}
K'=\mathop{\inf}_{x\in X} e^{-2w}\Big{[} K-{\Delta}w-(N-2)\Gamma(w)
-\mathop{\sup}_{f \in {\rm TestF}(M)}\frac{N-2}{\Gamma(f)}\big{(}\H_w(\nabla f, \nabla f)-\Gamma(w,f)^2\big{)}\Big{]}.
\end{eqnarray*}
Therefore, we can apply Theorem \ref{th-ricci} and finish the proof.

\end{proof}

Similarly, we have the following result.

\begin{corollary}\label{coro-conformal-2}
Let $M$ be a $\rcdkn$ space,  $M'=(X,\d',\mm')$ where $\d'=e^{w}\d$ and $\mm'=e^{v}\mm$, $v,w \in {\rm TestF}$.  Then for any $N'>\dim_{\rm max}(M)=\dim_{\rm max}(M')$, $M'$ satisfies the ${\rm RCD}^*(K', N')$ condition in case
\begin{eqnarray*}
K':=\mathop{\inf}_{x\in X} e^{-2w}\Big{[} K+\mathop{\sup}_{f \in {\rm TestF}(M)}\frac{1}{\Gamma(f)\,\d\mm}\big{(}e^{4w-v}{\bf Ricci}'_{N'}(\nabla' f, \nabla' f)-{\bf Ricci}_{N'}(\nabla f, \nabla f)\big{)}\Big{]}
\end{eqnarray*}
is a real number.
\end{corollary}

%%%%%%%%%%%%%%%%%%%%%%%%%%%%%%%%%%%%%%%%%%%%%%%%%%%%
\def\cprime{$'$} \def\cprime{$'$}

\end{document}